\documentclass[reqno,11pt]{amsart} \numberwithin{equation}{section}
\input{amssym.def}
\input{amssym.tex}
\begin{document}
\title[Weak Convergence Theorems for Generalized Hybrid mappings] {Weak convergence theorems for equilibrium
problems and generalized Hybrid mappings}
\author[S.Alizadeh,F. Moradlou,]{Sattar Alizadeh$^1$ and Fridoun Moradlou$^2$}
\address{\indent $^{1,2}$ Department of Mathematics
\newline \indent Sahand University of Technology
\newline \indent Tabriz, Iran}
\email{\rm $^1$ sa\_alizadeh@sut.ac.ir} \email{\rm $^2$
moradlou@sut.ac.ir \& fridoun.moradlou@gmail.com}
\thanks{}
\begin{abstract}
In this paper, we introduce a new modified Ishikawa iteration for
finding a common element of the set of solutions of an equilibrium
problem and the set of fixed points of generalized hybrid mappings
in a Hilbert space. Our results generalize, extend and enrich some
existing results in the literature.
\end{abstract}


\maketitle
\baselineskip=15.8pt \theoremstyle{definition}
  \newtheorem{df}{Definition}[section]
    \newtheorem{rk}[df]{Remark}
\theoremstyle{plain}
 \newtheorem{lm}[df]{Lemma}
  \newtheorem{thm}[df]{Theorem}
  \newtheorem{exa}[df]{example}
  \newtheorem{cor}[df]{Corollary}
  \setcounter{section}{0}
  \section{Introduction}
Throughout this paper, we denote by $\mathbb{N}$ and $\mathbb{R}$ the set of positive integers and real numbers,
respectively. Let $H$ be a real Hilbert space with inner product $\langle.,.\rangle$ and induced norm $\|.\|$, and
let $E$ be a nonempty closed  convex subset of $H$. Let $f$ be a bifunction from $E\times E$ to $\mathbb{R}$. The equilibrium problem for $f:E\times E\rightarrow\mathbb{R}$ is to find $x\in E$ such that
\begin{equation}\label{ep.1}
f(x,y)\geq0,\quad(y\in E).
\end{equation}
The set of solutions of (\ref{ep.1}) is denoted by $EP(f)$, i.e.,
$$EP(f)=\{x\in E:\;f(x,y)\geq0,\;\;\forall y\in E\}.$$
\par
 A self mapping  $S$ of $E$ is called nonexpansive if
$$\|Sx-Sy\|\leq\|x-y\|,\quad( x,y\in E).$$
We denote by $F(S)$ the set of fixed points of $S$.
\par
Let $S:E\longrightarrow H$ be a mapping and let $f(x,y)=\langle
Sx,y-x\rangle$ for all $x,y\in E$. Then $z\in EP(f)$ if only if
$\langle Sz,y-z\rangle\geq0$ for all $y\in E$, i.e., $z$ is a solution
of the variational inequality $\langle Sx,y-x\rangle\geq0$. So, the
formulation (\ref{ep.1}) includes variational inequalities as
special cases. Also, numerous problems in physics, optimization and
economics reduce to find a solution of (\ref{ep.1}). Some methods
have been proposed to solve the equilibrium problem; see for
instance, \cite{b,a4,e,a16,Kinder}.
\par
In the recent years, many authors studied the problem of finding a
common element of the set of fixed points of a nonexpansive mapping
and the set of solutions of an equilibrium problem in the framework
of Hilbert spaces and Banach spaces, respectively; see for instance,
\cite{am2,a2,a3,a5,a9,a10,a18,a20,a23,a24} and the references therein.
\par
Let $E$ be a nonempty closed convex subset of a Banach space. In
1953, for a self mapping $S$ of $E$, Mann \cite{mann} defined the
following iteration procedure:
\begin{equation}\label{mann}
\begin{cases}
 & x_{0} \in E \,\,\,\hbox{chosen arbitrarily},\\
 & x_{n+1}  = \alpha_{n}x_{n}+(1-\alpha_{n})Sx_{n},
\end{cases}
\end{equation}
where $0\leq \alpha_{n}\leq 1$ for all $n \in \mathbb{N}\cup \{0\}$,
\par
Let $K$ be a closed convex subset of a Hilbert space $H$.
 In 1974, for a Lipschitzian pseudocontractive self mapping $S$ of $K$, Ishikawa \cite{ishikawa}
defined the following iteration procedure:
\begin{equation}\label{ishikawa}
\begin{cases}
  & x_{0} \in K \,\,\,\hbox{chosen arbitrarily},\\
  & y_{n}= \beta_{n}x_{n}+ (1-\beta_{n})Sx_{n}, \\
  & x_{n+1}  = \alpha_{n}x_{n}+(1-\alpha_{n})Sy_{n},
\end{cases}
\end{equation}
where $0\leq \beta_{n} \leq \alpha_{n}\leq 1$ for all $n \in
\mathbb{N}\cup \{0\}$ and he proved strong convergence of the
sequence $\{x_{n}\}$ generated by the above iterative scheme if
$\lim_{n \to \infty} \beta_{n} =1$ and $\sum_{n=1}^{\infty}
(1-\alpha_{n})(1-\beta_{n}) = \infty$. By taking $\beta_{n} =1$ for
all $n \geq 0$ in $(\ref{ishikawa})$, Ishikawa iteration process
reduces to Mann iteration process.

\par
Process (\ref{ishikawa}) is indeed more general than process
(\ref{mann}). But research has been done  on the latter due probably
to reasons that the formulation of process (\ref{mann}) is simpler
than that of (\ref{ishikawa}) and that a convergence theorem for
process (\ref{mann}) may lead to a convergence theorem for process
(\ref{ishikawa}) provided that $\{\beta_{n}\}$ satisfies certain
appropriate conditions. On the other hand, the process (\ref{mann})
may fail to converge while process (\ref{ishikawa}) can still
converge for a Lipschitz pseudocontractive mapping in a Hilbert
space \cite{chidume}. Actually, Mann and Ishikawa iteration
processes have only weak convergence, in general (see \cite{genel}).
\par
In 2007, Tada and Takahashi \cite{a27} for finding an element of
$EP(f)\cap F(S)$, introduced the following iterative scheme  for a
nonexpansive self mapping $S$ of a nonempty, closed convex subset $E$ in a Hilbert space $H$:
\begin{equation*}
\begin{cases}
&x_{1} =x\in H\,\,\,\hbox{chosen arbitrarily},\\
 &u_n\in E\;\text{such that}\quad f(u_n,y) +\frac{1}{r_n}\langle y-u_n,u_n-x_n\rangle\geq0,\quad\forall\: y\in E\\
 & x_{n+1}=\alpha_n x_n+(1-\alpha_n)Su_n,
  \end{cases}
  \end{equation*}
  for all $n\in\mathbb{N}$, where $f: E\times E \to \mathbb{R}$ satisfies appropriate conditions,
  $\{\alpha_n\}\subset[a,b]$ for some $a,b\in(0,1)$ and $\{r_n\}\subset(0,\infty)$ satisfies
  $\liminf_{n\rightarrow\infty}r_n>0.$ They proved $\{x_n\}$ converges weakly to $w\in F(S)\cap EP(f)$, where
   $w=\lim_{n\rightarrow\infty}P_{F(S)\cap EP(f)}(x_n).$
   \par
Let $E$ be a nonempty closed convex subset of $H$. A self mapping
$S$ of $E$ is called \textit{generalized hybrid} \cite{Kocourek} if
there exist $\alpha, \beta \in \mathbb{R}$ such that
\begin{equation}
\alpha\|Sx - Sy\|^{2} + (1-\alpha)\|x-Sy\|^{2} \leq
\beta\|Sx-y\|^{2} + (1-\beta)\|x-y\|^{2}
\end{equation}
for all $x, y \in E$. We call such a mapping an\textit{ $(\alpha,
\beta)$-generalized hybrid} mapping.
   \par
In this paper, we modify Ishikawa iteration process for finding a common element of the set of solution of an
 equilibrium problem and the set of fixed points of a generalized hybrid mapping.
\section{Preliminaries}
 A self mapping $S$ of $E$ is called:
(i) \textit{firmly nonexpansive}, if
$\|Sx-Sy\|^{2}\leq \langle x-y,Sx-Sy\rangle$ for all $x,y \in E$;
(ii)\textit{ nonspreading}, if $2\|Sx-Sy\|^{2}\leq \|Sx-y\|^{2} +
\|Sy-x\|^{2}$ for all $x,y \in E$; (iii) \textit{hybrid}, if
$3\|Sx-Sy\|^{2}\leq \|x-y\|^{2}+\|Sx-y\|^{2} + \|Sy-x\|^{2}$ for all
$x,y \in E$. Also, a  self mapping $S$ of $E$ with $F (S ) \neq \emptyset$ is called \textit{quasi-nonexpansive}
if $\|x - Sy\| \leq \| x - y\|$ for all $x \in F(S)$ and $y \in E$. It is well-known that for a \textit{quasi-nonexpansive} mapping $S$,
$F(S)$ is closed and convex \cite{itoh}.
\par
 It easy to see that
\begin{itemize}
  \item $(1,0)$-generalized hybrid mapping is
nonexpansive mapping;
  \item $(2,1)$-generalized hybrid mapping is
nonspreading mapping;
  \item $(\frac{3}{2},\frac{1}{2})$-generalized hybrid mapping is hybrid
mapping.
\end{itemize}
\par
 Let $S$ be a generalized hybrid mapping. If $F(S)\neq\phi, $
then there exists $x\in E$ such that $x=Sx$, so for all $y\in E$ we
have
$$\alpha\|x - Sy\|^{2} + (1-\alpha)\|x-Sy\|^{2} \leq
\beta\|x-y\|^{2} + (1-\beta)\|x-y\|^{2}$$ and this yields that
$\|x-Sy\|\leq\|x-y\|$, i.e., an $(\alpha,\beta)$-generalized hybrid
mapping with $F(S)\neq\phi$, is quasi-nonexpansive.
\par
 We denote the weak convergence and the strong convergence of $\{x_n\}$ to $x\in H$ by $x_n\rightharpoonup x$ and $x_n\rightarrow x$,
 respectively and denote $\omega_\omega(x_n)$ the weak $\omega$-limit set of the sequence $\{x_n\}$, i.e., $\omega_{\omega}(x_n):=\{x\in H:\;\exists\{ x_{n_k}\}\subset\{x_n\} ;\quad x_{n_k}\rightharpoonup x\}.$
\par
Now, we recall some basic properties of Hilbert spaces which we will use in next section. For $x,y \in H$, we have
\begin{equation}\label{pre1}
   \|\alpha x+(1-\alpha)y\|^2=\alpha\|x\|^2+(1-\alpha)\|y\|^2 -\alpha(1-\alpha)\|x-y\|^2,\quad  \forall \alpha  \in
\mathbb{R},
\end{equation}
\begin{equation}\label{pre2}
  \|x+y\|^2 \leq\|x\|^2+2\langle y,x+y\rangle,
\end{equation}
and
\begin{equation}\label{pre3}
\|x-y\|^2=\|x\|^2-\|y\|^2-2\langle x-y,y\rangle.
\end{equation}
 Let $K$ be a closed convex subset of $H$ and let $P_K$
be metric (or nearest point) projection from $H$ onto $K$ (i.e., for
$x\in H,\:P_K x $ is the only point in $K$ such that
 $ \|x-P_Kx\|=inf\lbrace \|x-z\|:z\in H\rbrace$). Let $x\in H$ and $z\in K$, then $z=P_Kx$ if and only if:
\begin{equation}\label{metric projection}
\langle x-z,y-z\rangle\leq 0,
\end{equation}
for all $y \in K$. For more details we refer readers to \cite{Agarwal, T3}.
  \begin{lm}\cite{Yan}\label{lem.1}
Let $H$ be a Hilbert space
and $\lbrace x_n\rbrace$ be a sequence in $H$ such that there exists
a nonempty subset $E\subset H$ satisfying
\begin{itemize}
  \item [(i)] For every $u\in E,\;lim_{n\rightarrow\infty}\|x_n-u\| \;exists.$
  \item [(ii)] If a subsequence $\lbrace x_{n_j}\rbrace\subset \lbrace
x_n\rbrace$ converges weakly to $u$, then $u\in E$,
\end{itemize}
then there exists $x_0\in E$ such that $x_n\rightharpoonup x_0$.
\end{lm}
We will use the following lemmas in the proof of our main results in
next section.
\begin{lm}\cite{taka2003}\label{lem.2} Let $H$ be a Hilbert space and $E$ be a nonempty, closed and convex subset of $H$ and $\lbrace x_n\rbrace$ be a sequence in $H$.
If $\|x_{n+1}-x\|\leq\|x_n-x\|$ for all $n\in\mathbb{N}$ and $x\in
E$, then $\lbrace P_E(x_n)\rbrace$ converges strongly to some $z\in
E$, where $P_E$ stands for the metric projection on $H$ onto $E$.
\end{lm}
To study the equilibrium problem, for the bifunction $f:E\times E\longrightarrow\mathbb{R}$, we assume that $f$ satisfies the following conditions:
\begin{enumerate}
\item[(A1)] $f(x,x)=0$ for all $x\in E;$
\item[(A2)]$f$ is monotone, i.e., $f(x,y)+f(y,x)\leq0$ for all $x,y\in E;$
\item[(A3)]for each $x,y,z\in E$,
$$\hspace{-4cm}
\lim_{\downarrow0} f(tz+(1-t)x,y)\leq f(x,y);$$
\item[(A4)]for each $x\in E,\;y\mapsto f(x,y)$ is convex and lower semicontinuous.
\end{enumerate}
\par
The following lemma can be found in \cite{b}.
\begin{lm}\label{lm.3}
Let $E$ be a nonempty closed convex subset of $H$, let $f$ be a bifunction
from $E\times E$ to $\mathbb{R}$ satisfying $(A1)-(A4)$ and let $r > 0$ and $x\in H$. Then, there exists
$z\in E$ such that
$$f (z,y)+\frac{1}{r}\langle y-z,z-x\rangle\geq0,$$
for all $y\in E.$
\end{lm}
\par
The following lemma is established in \cite{a4}.
\begin{lm}\label{lm.4} For $r >0$, $x\in H$, define a mapping $T_r :H\longrightarrow E$ as follows:
$$\hspace{-2cm}T_r (x) = \{z\in E :f (z,y)+\frac{1}{r}\langle y-z,z-x\rangle\geq0,\; \forall y\in E\},$$
for all $x\in H$. Then, the following statements hold:

    \begin{enumerate}
      \item[(i)]$T_r$ is singel-valued;
      \item[(ii)]$T_r$ is firmly nonexpansive, i.e., for all $x,y\in H$,
      $$\|T_rx-T_ry\|^2\leq\langle T_rx-T_ry,x-y\rangle;$$
      \item[(iii)]$F(T_r)=EP(f);$
      \item[(iv)]EP(f) is closed and convex.
    \end{enumerate}
\end{lm}
\section{Main Results}
In this section, we prove weak the convergence theorems for finding a
common element of the set of solution of an equilibrium problem and
the set of fixed points of a generalized hybrid mapping.
\begin{thm}\label{thm1}
 Let $E$ be a nonempty closed convex subset of a real Hilbert space $H$. Let $f$ be a bifunction from $E\times E$ to $\mathbb{R}$ satisfying $(A1)-(A4)$ and $S$ be a generalized hybrid mapping of $E$  to $H$ with $F(S)\cap EP(f)\neq\phi$.
Assume that $0<\alpha\leq\alpha_n\leq1$  and $\{r_n\}\subset(0,\infty)$ satisfies
$\liminf_{n\rightarrow\infty}r_n>0$ and  $\{\beta_n\}$ is sequence
in $[b,1]$ for some $b\in(0,1)$ such that $\liminf_{n\rightarrow\infty}\beta_n(1-\beta_n)>0$. If  $\{x_n\}$ and $\{u_n\}$ be sequences generated by $x=x_1\in H$ and
\begin{equation*}
\begin{cases}
 &u_n\in E\;\text{such that}\quad f(u_n,y) +\frac{1}{r_n}\langle y-u_n,u_n-x_n\rangle\geq0,\quad\forall\: y\in E\\
 & y_n=(1-\beta_n)x_n+\beta_nSu_n,\\
 & x_{n+1}=(1-\alpha_n)x_n+\alpha_nSy_n,
  \end{cases}
\end{equation*}
for all $n\in\mathbb{N}$. Then $x_n\rightharpoonup v\in F(S)\cap EP(f)$, where $v=lim_{n\to\infty}P_{F(S)\cap EP(f)}(x_n)$.
\end{thm}
\begin{proof}
 By Lemma \ref{lm.3}, $\{u_n\},\{y_n\}$  and $\{x_n\}$ are well defined. Since $S$ is a generalized hybrid mapping such that  $F(S)\neq\phi$,  $S$  is quasi-nonexpansive.
So $F(S)$ is closed and convex. Also by hypothesis $EP(f)\neq\phi$ . Set $q\in F(S)\cap EP(f).$
\par
From $u_n=T_{r_n}x_n$, we get
\begin{equation}\label{inq.1}
\|u_n-q\|=\|T_{r_n}x_n-T_{r_n}q\|\leq\|x_n-q\|.
\end{equation}
 On the other hand
 \begin{equation}\label{lm1.1}\begin{aligned}
    \|y_n-q\|^2
    &=(1-\beta_n)\|x_n-q\|^2+\beta_n\|Su_n-q\|^2-\beta_n(1-\beta_n)\|x_n-Su_n\|^2\\
    &\leq(1-\beta_n)\|x_n-q\|^2+\beta_n\|x_n-q\|^2-\beta_n(1-\beta_n)\|x_n-Su_n\|^2\\
    &=\|x_n-q\|^2-\beta_n(1-\beta_n)\|x_n-Su_n\|^2 ,
\end{aligned}
\end{equation}
and hence
\begin{equation}\label{inq1}
\begin{aligned}
    \|x_{n+1}-q\|^2
    &=\|(1-\alpha_n)x_n+\alpha_nSy_n-q\|^2\\
    &=(1-\alpha_n)\|x_n-q\|^2+ \alpha_n\|Sy_n-q\|^2-\alpha_n(1-\alpha_n)\|x_n-Sy_n\|^2\\
    &\leq(1-\alpha_n)\|x_n-q\|^2+\alpha_n\|y_n-q\|^2-\alpha_n(1-\alpha_n)\|x_n-Sy_n\|^2\\
    &\leq(1-\alpha_n)\|x_n-q\|^2+\alpha_n\|x_n-q\|^2-\alpha_n\beta_n(1-\beta_n)\|x_n-Su_n\|^2\\
    &\hspace{6.1cm}-\alpha_n(1-\alpha_n)\|x_n-Sy_n\|^2\\
    &\leq\|x_n-q\|^2-\alpha_n\beta_n(1-\beta_n)\|x_n-Su_n\|^2\\
    &\leq\|x_n-q\|^2.
\end{aligned}
\end{equation}
So, we can conclude that  $lim_{n\rightarrow\infty}\|x_n-q\|$
exists. This yields that $\{x_n\}$ and $\{y_n\}$ are
bounded. It follows from $(\ref{inq1})$ that
$$\|x_{n+1}-q\|^2\leq\|x_n-q\|^2-\alpha_n\beta_n(1-\beta_n)\|x_n-Su_n\|^2.$$
Using  $0<\alpha\leq\alpha_n\leq1,$
it is easy to see that
$$\|x_{n+1}-q\|^2\leq\|x_n-q\|^2-\alpha\beta_n(1-\beta_n)\|x_n-Su_n\|^2.$$
Also, we have
$$ 0\leq\alpha\beta_n(1-\beta_n)\|x_n-Su_n\|^2\leq\|x_n-q\|^2-\|x_{n+1}-q\|^2\rightarrow0 ,$$
as $n\rightarrow\infty$, since
$\liminf_{n\to\infty}\beta_n(1-\beta_n)>0$. Therefore
\begin{equation}\label{eq1}
\|x_n-Su_n\|\longrightarrow 0.
\end{equation}
This yields that
\begin{equation}\label{eq2}
\|y_n-x_n\|=\beta_n\|x_n-Su_n\|\longrightarrow 0.
\end{equation}
 Using  $(\ref{pre3})$ and Lemma \ref{lm.4} , we get
\begin{equation*}\begin{aligned}
\|u_n-q\|^2&=\|T_{r_n}x_n-T_{r_n}q\|^2\\
&\leq\langle T_{r_n}x_n-T_{r_n}q,x_n-q\rangle\\
&=\langle u_n-q,x_n-q\rangle\\
&=\frac{1}{2}(\|u_n-q\|^2+\|x_n-q\|^2-\|x_n-u_n\|^2),
\end{aligned}
\end{equation*}
hence
$$\|u_n-q\|^2\leq\|x_n-q\|^2-\|x_n-u_n\|^2.$$
Then, by the convexity of $\|.\|^2$, we have
\begin{equation*}\begin{aligned}
\|y_n-q\|^2 &= \|(1-\beta_n)(x_n-q)+\beta_n(Su_n-q)\|^2\\
&\leq(1-\beta_n)\|x_n-q\|^2+\beta_n\|Su_n-q\|^2\\
&\leq(1-\beta_n)\|x_n-q\|^2+\beta_n\|u_n-q\|^2\\
&\leq(1-\beta_n)\|x_n-q\|^2+\beta_n(\|x_n-q\|^2-\|x_n-u_n\|^2)\\
&=\|x_n-q\|^2-\beta_n\|x_n-u_n\|^2.
\end{aligned}
\end{equation*}
Therefore
\begin{equation}\label{inq.8}
\beta_n\|x_n-u_n\|^2\leq\|x_n-q\|^2-\|y_n-q\|^2
\end{equation}
Since $\{\beta_n\}\subset [b,1]$, it follows from (\ref{inq.8}) that
\begin{equation*}\begin{aligned}
b\|x_n-u_n\|^2&\leq\beta_n\|x_n-u_n\|^2\\
&\leq\|x_n-q\|^2-\|y_n-q\|^2\\
&=(\|x_n-q\|-\|y_n-q\|)(\|x_n-q\|+\|y_n-q\|)\\
&\leq\|y_n-x_n\|(\|x_n-q\|+\|y_n-q\|)
\end{aligned}
\end{equation*}
Using the boundedness of $\{x_n\}$ and $\{y_n\}$, it follows from (\ref{eq2}) and the above inequality that
\begin{equation}\label{eq.6}
\lim_{n\rightarrow\infty}\|x_n-u_n\|=0.
\end{equation}
Since $\liminf_{n\rightarrow\infty}r_n>0$, we get
\begin{equation}\label{eq.8}
\lim_{n\rightarrow\infty}\Big\|\frac{x_n-u_n}{r_n}\Big\|=\lim_{n\rightarrow\infty}\frac{1}{r_n}\|x_n-u_n\|=0.
\end{equation}
As $\beta_nSu_n=y_n-(1-\beta_n)x_n$, we have
\begin{equation*}\begin{aligned}
b\|u_n-Su_n\|&\leq\beta_n\|u_n-Su_n\|=\|y_n-(1-\beta_n)x_n-\beta_n u_n\|\\
&\leq\|y_n-x_n\|+\beta_n\|x_n-u_n\|\\
&\leq\|y_n-x_n\|+\|x_n-u_n\|\\
\end{aligned}
\end{equation*}
From (\ref{eq2}) and (\ref{eq.6}), we obtain
\begin{equation}\label{eq.7}
\lim_{n\rightarrow\infty}\|u_n-Su_n\|=0.
\end{equation}
Since $\{x_n\}$ is bounded, there exists a subsequence $\{x_{n_i}\}$ of $\{x_n\}$ such that $x_{n_i}\rightharpoonup u$. By (\ref{eq.6}) we obtain $u_{n_i}\rightharpoonup u$. We know that $E$ is closed and convex and  $\{u_{n_i}\}\subset E$, therefore $u\in E$.
\par
Now, we show that $u\in F(S)\cap EP(f)$. Since $u_n=T_{r_n}x_n$, we get
$$f(u_n,y)+\frac{1}{r_n}\langle y-u_n,u_n-x_n\rangle\geq0,$$
for all $y\in E$. From the condition $(A2)$, we obtain
$$\frac{1}{r_n}\langle y-u_n,u_n-x_n\rangle\geq f(y,u_n),$$
for all $y\in E$, therefore
\begin{equation}\label{inq.7}
\Big\langle y-u_{n_i},\frac{u_{n_i}-x_{n_i}}{r_{n_i}}\Big\rangle\geq f(y,u_{n_i}),
\end{equation}
for all $y\in E$. It follows from (\ref{eq.8}), (\ref{inq.7}) and condition $(A4)$ that
$$0\geq f(y,u),$$
for all $y\in E$. Suppose that $t\in(0,1]$, $y\in E$ and  $y_t=ty+(1-t)u.$ Therefore, $y_t\in E$ and so $f(y_t,u)\leq0$. Hence
$$0=f(y_t,y_t)\leq tf(y_t,y)+(1-t)f(y_t,u)\leq tf(y_t,y),$$
and dividing by t, we have $f(y_t,y)\geq0$, for all $y\in E$. By taking the limit as $t\downarrow0$ and using $(A3)$, we get $u\in EP(f)$.
\par
Next we show that  $u\in F(S)$. Since $S$ is a generalized
hybrid mapping, then
$$\gamma\|Sx-Sy\|^2+(1-\gamma)\|x-Sy\|^2\leq\lambda\|Sx-y\|^2+(1-\lambda)\|x-y\|^2$$
hence
\begin{equation*}
0\leq\lambda\|Sx-y\|^2+(1-\lambda)\|x-y\|^2-\gamma\|Sx-Sy\|^2-(1-\gamma)\|x-Sy\|^2
\end{equation*}
replacing $x$ and $y$ by $u_{n}$ and $u$ in  above inequality,
respectively, we get
\begin{equation*}
\begin{aligned}
0 &\leq
\lambda(\|Su_n\|^2-2\langle Su_n,u\rangle+\|u\|^2)+(1-\lambda)(\|u_n\|^2-2\langle u_n,u\rangle+\|u\|^2)\\
& \hspace{0.3cm}-\gamma(\|Su_n\|^2-2\langle Su_n,Su\rangle+\|Su\|^2) -(1-\gamma)(\|u_n\|^2-2\langle u_n,Su\rangle+\|Su\|^2)\\
&=\|u\|^2-\|Su\|^2+(\lambda-\gamma)(\|Su_n\|^2-\|u_n\|^2)\\&\hspace{0.3cm}+2\gamma\langle Su_n-u_n,Su\rangle -2\lambda\langle Su_n-u_n,u\rangle+2\langle u_n,Su-u\rangle\\
&\leq \|u\|^2-\|Su\|^2+(\lambda-\gamma)(\|Su_n\|+\|u_n\|)(\|Su_n - u_n\|) \\
&\hspace{0.3cm}+2\gamma\langle Su_n-u_n,Su\rangle -2\lambda\langle Su_n-u_n,u\rangle+2\langle u_n,Su-u\rangle.
\end{aligned}\end{equation*}
Now, substituting $n$ by $n_{i}$, we have
\begin{equation}\begin{aligned}\label{inq2}
&0\leq\|u\|^2-\|Su\|^2+(\lambda-\gamma)(\|Su_{n_i}\|+\|u_{n_i}\|)(\|Su_{n_i} - u_{n_i}\|) \\
&\hspace{0.3cm}+2\gamma\langle Su_{n_i}-u_{n_i},Su\rangle -2\lambda\langle Su_{n_i}-u_{n_i},u\rangle+2\langle u_{n_i},Su-u\rangle.
\end{aligned}\end{equation}
for all $ i\in \mathbb{N}$. Since  $u_{n_i}\rightharpoonup u$  as
$i\rightarrow\infty$, it follows from (\ref{eq.7}) and
(\ref{inq2}) that
\begin{equation*}\begin{aligned}
0&\leq\|u\|^2-\|Su\|^2+2\langle u,Su-u\rangle\\
 &=2\langle u,Su\rangle-\|u\|^2-\|Su\|^2\\&=-\|u-Su\|^2.
\end {aligned}\end{equation*}
 So, we have  $Su=u$, i.e.,  $u\in F(S)$. Therefore
 the condition (ii) of Lemma \ref{lem.1} satisfies for $E=F(S)\cap EP(f)$. On the other hand, we see that $ \lim_{n\rightarrow \infty}\|x_n-q\|$  exists for $q\in F(S)\cap EP(f)$. Hence, it follows
 from Lemma \ref{lem.1}  that there exists  $v\in F(S)\cap EP(f)$  such that $ x_n\rightharpoonup v$. In addition, for all  $q\in F(S)\cap EP(f)$, we have
 $$\|x_{n+1}-q\|\leq\|x_n-q\|,\qquad\forall~n\in \mathbb{N}, $$
so, Lemma \ref{lem.2} implies that there exists some  $w\in F(S)\cap EP(f)$
such that  $P_{F(T)\cap EP(f)}(x_n)\rightarrow w$. Then
$$\Big\langle v-P_{F(T)\cap EP(f)}(x_n),x_n-P_{F(T)\cap EP(f)}(x_n)\Big\rangle\leq0.$$
Hence, we get
$$\langle v-w,v-w \rangle =\|v-w\|^2\leq0.$$Therefore
$v=w, $ i.e.,  $x_n\rightharpoonup v=lim_{n\to\infty}P_{F(T)\cap EP(f)}(u_n)$.
\end{proof}
\begin{cor}
Let $E$ be a nonempty closed convex subset of a real Hilbert space $H$ and  $S$ be a generalized hybrid mapping of $E$  to $H$ with $F(S)\neq\phi$. Assume that $0<\alpha\leq\alpha_n\leq1$ and  $\{\beta_n\}$ is sequence
in $[b,1]$ for some $b\in(0,1)$ such that $\liminf_{n\rightarrow\infty}\beta_n(1-\beta_n)>0$. If  $\{x_n\}$ and $\{u_n\}$ be sequences generated by $x=x_1\in H$ and
\begin{equation*}
\begin{cases}
 &u_n\in E\;\text{such that}\quad\langle y-u_n,u_n-x_n\rangle\geq0,\quad\forall\: y\in E\\
 & y_n=(1-\beta_n)x_n+\beta_nSu_n,\\
 & x_{n+1}=(1-\alpha_n)x_n+\alpha_nSy_n,
  \end{cases}
\end{equation*}
for all $n\in\mathbb{N}$. Then $x_n\rightharpoonup v\in F(S)$, where $v=lim_{n\to\infty}P_{F(S)}(x_n)$.
\end{cor}
\begin{proof}
Letting $f(x,y)=0$ for all $x,y\in E$ and $r_n=1$ for all $n\in\mathbb{N}$ in Theorem \ref{thm1}, we get the desired result.
\end{proof}
\begin{cor}
Let $E$ be a nonempty closed convex subset of a real Hilbert space $H$. Let $f$ be a bifunction from $E\times E$ to $\mathbb{R}$ satisfying $(A1)-(A4)$ and $S$ be a generalized hybrid mapping of $E$  to $H$ with $F(S)\cap EP(f)\neq\phi$.
Assume that  $\{r_n\}\subset(0,\infty)$ satisfies
$\liminf_{n\rightarrow\infty}r_n>0$ and  $\{\beta_n\}$ is sequence
in $[b,1]$ for some $b\in(0,1)$ such that $\liminf_{n\rightarrow\infty}\beta_n(1-\beta_n)>0$. If  $\{x_n\}$ and $\{u_n\}$ be sequences generated by $x=x_1\in H$ and
\begin{equation*}
\begin{cases}
 &u_n\in E\;\text{such that}\quad f(u_n,y) +\frac{1}{r_n}\langle y-u_n,u_n-x_n\rangle\geq0,\quad\forall\: y\in E\\
 & x_{n+1}=S((1-\beta_n)x_n+\beta_nSu_n),
  \end{cases}
\end{equation*}
for all $n\in\mathbb{N}$. Then $x_n\rightharpoonup v\in F(S)\cap EP(f)$, where $v=lim_{n\to\infty}P_{F(S)\cap EP(f)}(x_n)$.
\end{cor}
\begin{proof}
Letting $\alpha_n=1$ for all $n\in\mathbb{N}$, in Theorem \ref{thm1}, we get the desired result.
\end{proof}
\begin{thm}
 Let $E$ be a nonempty closed convex subset of a real Hilbert space $H$. Let $f$ be a bifunction from $E\times E$ to $\mathbb{R}$ satisfying $(A1)-(A4)$ and $S$ be a hybrid mapping of $E$  to $H$ with $F(S)\cap EP(f)\neq\phi$.
Assume that $0<\alpha\leq\alpha_n\leq1$  and $\{r_n\}\subset(0,\infty)$ satisfies
$\liminf_{n\rightarrow\infty}r_n>0$ and  $\{\beta_n\}$ is sequence
in $[b,1]$ for some $b\in(0,1)$ such that $\liminf_{n\rightarrow\infty}\beta_n(1-\beta_n)>0$. If  $\{x_n\}$ and $\{u_n\}$ be sequences generated by $x=x_1\in H$ and
\begin{equation*}
\begin{cases}
 &u_n\in E\;\text{such that}\quad f(u_n,y) +\frac{1}{r_n}\langle y-u_n,u_n-x_n\rangle\geq0,\quad\forall\: y\in E\\
 & y_n=(1-\beta_n)x_n+\beta_nSu_n,\\
 & x_{n+1}=(1-\alpha_n)x_n+\alpha_nSy_n,
  \end{cases}
\end{equation*}
for all $n\in\mathbb{N}$. Then $x_n\rightharpoonup v\in F(S)\cap EP(f)$, where $v=lim_{n\to\infty}P_{F(S)\cap EP(f)}(x_n)$.
\end{thm}
\begin{proof}
Since $S$ is a hybrid mapping, hence $S$ is a $(\frac{3}{2},\frac{1}{2})$-generalized hybrid mapping. Therefore by Theorem \ref{thm1}, we get the desired result.
\end{proof}
\begin{rk}
Since nonexpansive mappings are  $(1,0)$- generalized
hybrid mappings and nonspreading mappings are $(2,1)$-generalized hybrid mappings, then the Theorem \ref{thm1} holds for these mappings.
\end{rk}

\end{document}